\documentclass[10pt]{article}

\usepackage{amssymb} 
\usepackage{latexsym} 
\usepackage{amsmath}
\usepackage{mathtools} 
\usepackage{thmtools}
\usepackage{amsthm}
\usepackage{tabls} 
\usepackage{graphicx}
\usepackage{overpic}
\usepackage{subcaption}
\usepackage[english]{babel}
\usepackage{a4wide}
\usepackage[utf8]{inputenc}
\usepackage{color}
\usepackage{epsfig}
\usepackage{enumerate}
\usepackage{enumitem}
\usepackage{accents}

\newtheorem{theorem}{Theorem}

\newtheorem{lemma}{Lemma}
\newtheorem{example}{Example}

\newenvironment{theoremproof}[1]{\par\noindent\textbf{Proof of Theorem}\space\textbf{#1.}\!}{\hfill $\Box$\mybreak\noindent}

\def\cadre{$$\vcenter\bgroup\advance\hsize by -2em\noindent
	\refstepcounter{equation}(\theequation)~\ignorespaces}
\makeatletter
\def\endcadre{\egroup\eqno$$\global\@ignoretrue}

\newcommand{\mybreak} {\par\vspace{2mm}\noindent}

\makeatletter
\def\imod#1{\allowbreak\mkern10mu({\operator@font mod}\,\,#1)}
\makeatother

\newcommand{\comment}[1]{}

\newcommand{\kn}[2] {#1^{\underline{#2}}}

\newcommand{\pr} {{\rm Pr}}
\newcommand{\pra}[1] {\pr\left\{#1\right\}}

\newcommand{\E} {\mathbb{E}}
\newcommand{\R} {\mathbb{R}}

\newcommand{\Ea}[1] {\E\left(#1\right)}

\title{Second-order moments of the size of randomly induced subgraphs of given order} 
\author{AA}
\date{}
\author{Nicola Apollonio\footnote{Istituto per le Applicazioni del
		Calcolo, M. Picone, Via dei Taurini 19, 00185 Roma, Italy.
		\texttt{nicola.apollonio@cnr.it}}
}

\begin{document}
\pagestyle{plain}
\maketitle	
	\begin{abstract}
	For a graph $G$ and a positive integer $c$, let $M_c(G)$ be the size of a subgraph of $G$ induced by a randomly sampled subset of $c$ vertices. Second-order moments of $M_c(G)$ encode part of the structure of $G$. We use this fact, coupled to classical moment inequalities, to prove graph theoretical results, to give combinatorial identities, to bound the size of the $c$-densest subgraph from below and the size of the $c$-sparsest subgraph from above, and to provide bounds for approximate enumeration of trivial subgraphs. 	
\end{abstract}
\mybreak

\section{Introduction}\label{sec:intro}
Consider the following three seemingly unrelated Graph Theoretical results:
\begin{theorem}[\cite{hararyetal}, \cite{bosak}, \cite{siran}]\label{thm:bh} Let $G$ be a graph on $n$ vertices and let $c$ be a positive integer such that $2\leq c\leq n-2$. If all the induced subgraphs of $G$ of order $c$ have the same size, then $G$ is either complete or co-complete.
\end{theorem}
Here co-complete means the complement of a complete graph.
\begin{theorem}\label{thm:fre_1} Let $G$ be a graph of order $n$ and positive size $m$. If $c$ is an integer number such that $2\leq c\leq n-1$, then $G$ has a subgraph of order $c$ whose size is at least 
	$$\frac{c-2}{n-2}\left\{\frac{n-c}{n-3}\left(\sum_{v\in V(G)}\frac{d(v)(d(v)-1)}{m}\right)+\left(1-\frac{n-c}{n-3}\right)(m-1)\right\}+1.$$ Furthermore, the bound is attained precisely when
	\begin{itemize}
		\item[--] $G$ is any graph with $m\geq 1$ if $c=2$;
		\item[--] $G$ is either a complete bipartite graph or $G$ consists of $t$ disjoint edges, $1\leq t\leq \lfloor n/2 \rfloor$ if $c=3$;
		\item[--] $G$ is a complete graph, or $G$ is a star, or $G$ has exactly one edge
		if $4\leq c\leq n-1$.
	\end{itemize}  
	If $c=n$, then the bound trivially says that $G$ has size $m$.
\end{theorem}
Here $d(v)$ denotes the degree of vertex $v\in V(G)$.
\begin{theorem}[\cite{linlin}, \cite{prodtichy}]\label{thm:count_stable} Every tree of order $n$ has at most $1+2^{n-1}$ independent sets. Furthermore, the bound is attained precisely when the tree is a star on $n-1$ leaves. 
\end{theorem}
Here the empty subset of vertices of a tree is counted as an independent set.
\mybreak
The common thread that runs across the results above is the second order moment structure of the random variable $M_c(G)$, defined as the size of a subgraph of $G$ induced by a randomly sampled set of $c$ vertices. Remark that we are drawing vertices without replacement. A formula for the first two moments of $M_c(G)$, which implies a linear-time algorithm for calculating these moments, was provided in \cite{our}. The formula was originally motivated by the need to assess whether the observed edge-density of the subgraphs induced by the communities of a large network were statistically significant compared to their expected values under a suitable null model. In this paper we give a different presentation of the formula. The random variable $M_c(G)$ encodes a variety of information about the graph $G$. For instance, the support $R_c(G)$ of $M_c(G)$, namely, the set of values attained with positive probability by $M_c(G)$, is the set of different sizes an induced subgraph of $G$ of order $c$ can have. The related set $\cup_c R_c(G)$, called the \emph{range of subgraph size of $G$} in \cite{cfm}, was first considered by Erd\H{o}s and McKay (see \cite{alonrtal}), who raised a conjecture on the maximum width of an interval of the form $\{0,1,\ldots,u\}$ contained in the range of subgraph size of $G$. They also provided a bound on $u$ in terms of the maximum number of induced \emph{trivial} subgraphs of $G$. Recall that an induced subgraph is \emph{trivial} if it is either complete or co-complete. The bound was considerably improved in \cite{alonrtal}, and the conjecture was solved for random graphs in \cite{cfm}. In the light of this discussion, Theorem \ref{thm:bh} characterizes those graphs $G$ such that $R_c(G)$ is the singleton $\{u(c)\}$, namely, those graphs $G$ such that $M_c(G)$ is constant almost surely. Since a random variable is constant almost surely if and only if its variance is zero, it follows that Theorem \ref{thm:bh} actually characterizes those graphs $G$ such that the variance of $M_c(G)$ is zero. Theorem \ref{thm:bh} is considerably extended in \cite{idue} using tools from extremal graph theory. Theorem \ref{thm:fre_1} can also be viewed as a statement about $R_c(G)$: it is a statement about the maximum $u(c)$ of $R_c(G)$ in fact. Such a maximum is known as the size of the $c$-densest subgraph \cite{kp}. Since by the second order Fréchet inequality (see \cite{hoppeseneta} and Section \ref{sec:prel}), provided that $M_c(G)$ is not almost surely zero and with $\Ea{\cdot}$ denoting expectation, it holds that
$$\frac{\Ea{M_c(G)(M_c(G)-1)}}{\Ea{M_c(G)}}+1\leq u(c),$$
Theorem \ref{thm:fre_1} follows rather straightforwardly by such an inequality once we have computed the expectations occurring therein---we accomplish this task in Theorem \ref{thm:main}---. Actually, when the minimum $\ell(c)$ of $M_c(G)$ in $R_c(G)$ is non zero, it can be shown (see Theorem \ref{thm:listin}) that Fréchet second order inequality is strengthened by the inequality
\begin{equation}\label{eq:bada}
\sigma^2_c(G)\leq \left\{ u(c)-\Ea{M_c(G)}\right\}\left\{\Ea{M_c(G)}-\ell(c)\right\}
\end{equation} 
due to Bhatia and Davis \cite{bada} for a discrete finitely supported random variable. Here $\sigma^2_c(G)$ denotes the variance of $M_c(G)$. We have learned of this inequality in \cite{itredibada} which also seems to be the first place where the inequality emerged within the Graph Theoretical community. While the assertion that Bathia-Davis inequality strengthens the  second order Fréchet inequality certainly needs a proof (see Theorem \ref{thm:listin}), it is otherwise apparent that the former sharpens the latter because it uses more information about the support of $M_c(G)$. Notice indeed that $\ell(c)=0$ if and only if the independence number $\alpha(G)$ of $G$ is at least $c$. In general, 
\begin{equation}\label{eq:fund_ell}
	c-\alpha(G)\leq \ell(c)
\end{equation}
because, if $c\geq \alpha(G)$, then any set $S$ of $c$ vertices induces at least $c-\alpha(G)$ edges: if the number $h$ of edges induced by $S$ were less than $c-\alpha(G)$, then by deleting at most $h$ vertices from $S$ we would be left with a stable set $S'$ of cardinality at least $\alpha(G)+1$. Therefore, the bound given in Theorem \ref{thm:fre_1} can be sharpened for all those $c$ such that $c\geq \alpha(G)$. 
\mybreak
Besides the support $R_c(G)$ of $M_c(G)$ also the sets $T(k)=\{r\in R_c(G)\ |\ r\geq k\}$, $k\in \mathbb{N}\cup\{0\}$, are quite informative. These sets are called 
the \emph{tails} of $M_c(G)$. For instance, if $G$ is non bipartite, $\pra{M_c(G)=\kn{c}{2}/2}$
is the fraction of cliques of $G$ on $c$ vertices and $\left( n \atop c\right)\pra{M_c(G)=\kn{c}{2}/2}$ is their count. Clearly, $\pra{M_c(G)=\kn{c}{2}/2}=\pra{M_c\in T(\kn{c}{2}/2)}$. The same remark applies to the fraction of subtrees of order $c$ of a given tree or even to the balanced bicliques of order $c$ in a bipartite graph. Likewise, $\pra{M_c(G)=0}$ is the fraction of stable sets and $\pra{M_c(G)=0}=1-\pra{M_c(G)\in T(1)}$. Since there are several known upper bounds on the tail probability $\pra{M_c(G)\in T(k)}$, $k\in \mathbb{N}\cup\{0\}$ and, analogously, several known lower bounds on $\pra{M_c(G)\in T(1)}$, both type of bounds being expressed in terms of the second moments of $M_c(G)$, it follows that once we know the second order moment structure of $M_c(G)$, we can use any of these bounds to give upper bounds on the number of trivial subgraphs of $G$. For instance, a version of the celebrated Chung-Erd\H{o}s inequality (see \cite{chungerdos},\cite{hoppeseneta}, \cite{petrov_0},\cite{petrov}) reads as follows
\begin{equation}\label{eq:erdoschung_1}
	\pra{M_c(G)\geq 1}\geq \frac{\left(\Ea{M_c(G)}\right)^2}{\Ea{M_c(G)^2}}
\end{equation}
and specializes both Petrov's and Paley-Zygmund inequalities (see again \cite{petrov_0}). Since the second moment of a random variable is always at least the square of its first moment, it follows that 
$$\left( n \atop c\right)\left(1-\frac{\left(\Ea{M_c(G)}\right)^2}{\Ea{M_c(G)^2}}\right)$$
is always a valid upper bound on the number of the independent sets with $c$-vertices. 
\mybreak
The rest of the paper goes as follows. In Section \ref{sec:prel}, we give some background on the moment inequalities we use throughout the paper. In Section \ref{sec:main}, we compute the second order moments of $M_c(G)$ and give examples. Section \ref{sec:res} is devoted to the proofs of the results presented here and to further applications of the second order moment structure of $M_c(G)$.     

\section{Preliminaries}\label{sec:prel}
 For a positive integer $t$ the symbol $[t]$ is the set of the first $t$ positive integers. For a non negative integer $s$, we denote by $\kn{t}{s}$ the $s$-th decreasing factorial of $t$, namely the product $t(t-1)\cdots(t-s+1)$. If $s>t$, then $\kn{t}{s}=0$. Notice that $\kn{t}{0}=1$, $\kn{t}{t}=t!$ and that ${t \choose s}=\kn{t}{s}/s!$.
We deal with discrete random variables with finite support $R$ contained in an interval $\{\ell,\ell+1,\ldots,u\}$ of non negative integers. If $M$ is any such random variable and $r\in R$, then we denote by $p_r$ the probability $\pra{M=r}$. We denote $\Ea{M}$ by $S_1$ and $\frac{1}{2}\Ea{M(M-1)}$ by $S_2$. These are called the first two \emph{binomial moments} of $M$. We denote $\Ea{M^2}$ by $\mu_2$ and $\mu_2-S_1$ by $\sigma^2$ (the variance of $M$).
One has $2S_2+S_1=\mu_2$ and $\sigma^2=2S_2-S_1(S_1-1)$. We resort to such a notation because the tail inequalities we are interested in are typically stated in this terminology due to its relation with Bonferroni's inequalities and the Inclusion/Exclusion Principle (see \cite{hoppeseneta}, and \cite{sachov}). When $M=M_c(G)$, then we write $S_1(c)$, $S_2(c)$, $\mu_2(c)$ and $\sigma^2(c)$ for the corresponding moments.  
\mybreak
The following lemma is a special case of the so-called Fréchet identities (see \cite{sachov}), which we will use in the next section to compute the first two moments of $M_c(G)$. 
\begin{lemma}\label{lem:frechet}
	Let $\mathcal{A}=\{A_e \ |\ e\in E\}$ be a finite collection of events in some probability space $\Psi$. If $M$ is the random variable defined as the number of events which occur among those of $\mathcal{A}$, namely $M=\sum_eI(A_e)$, where $I(\cdot)$ is the indicator function, then
	$$S_1=\sum_{e\in E}\pra{A_e}\,\, \text{and} \,\,\, S_2=\sum_{\substack{\{e,f\}\subseteq E\\ e\not=f}}\pra{A_e\wedge A_f}.$$ 
\end{lemma}

\begin{theorem}\label{thm:listin}
Let $M$ be a random variable such that $R\subseteq \{\ell,\ell+1,\ldots,u\}\subseteq \mathbb{N}\cup \{0\}$ with, possibly, $\ell=u$. Then
\begin{enumerate}[label={\rm (\roman*)}]
	\item\label{com:i} \textsc{Second order Fréchet inequality}, $(u-1)S_1\geq 2S_2$ and equality is attained if and only if $\#R=1$ or $\#R=2$ and $\ell=0$.
	\item\label{comm:ii} \textsc{Bathia-Davis inequality}, $\sigma^2\leq (u-S_1)(S_1-\ell)$ and equality is attained if and only if $\#R\leq 2$. 
	\item\label{comm:iii} \textsc{Chung-Erd\H{o}s inequality}, $\pra{M\geq 1}\leq \frac{S_1^2}{2S_2+S_1}$.
	\item\label{com:iv} \textsc{Petrov's inequality}, $\pra{M\geq t}\leq \frac{(S_1-t+1)^2}{2S_2+S_1}$ for any $t$ such that $1\leq t\leq S_1$.
	\item\label{com:v} \textsc{Second order factorial moment inequality}, $\pra{M\geq t}\leq \frac{2S_2}{\kn{t}{2}}$ for $t\geq 1$, $t$ integer. 
\end{enumerate}
Moreover, the Bathia-Davis inequality and the second order Fréchet inequality are equivalent whenever $\ell=u$ or $\ell=0$ while, if $0<\ell<u$, then the former dominates the latter.   	
\end{theorem}
\begin{proof}
All the inequalities in the statement above, but the last one, can be found in \cite{hoppeseneta} and almost all of them can be obtained via linear programming duality (see \cite{bopre}). The second order factorial inequality in \ref{com:v} follows straightforwardly by Markov's inequality once we observe that for a positive integer $t$ one has $M\geq t\Longrightarrow M(M-1)\geq t(t-1)$. The Chung-Erd\H{o}s inequality was originally written as 
$$\pra{\bigvee_e A_e}\geq \left. \left(\sum_{e\in E}\pra{A_e}\right)^2 \middle / \sum_{(e,f)\in E^2}\pra{A_e\bigwedge A_f}\right.$$
for arbitrary events $A_e$, $e\in E$ in a probability space \cite{chungerdos}, and it assumes the form given in \ref{comm:iii} after Lemma \ref{lem:frechet} by defining $M$ as the number of events among the $A_e$'s which occur. It remains to prove the tightness of the inequality in \ref{com:i} and the relation between this inequality and the inequality in \ref{comm:ii}. We prove both facts at the same time. Since $\sigma^2=2S_2-S_1(S_1-1)$, if $\ell=0$, then the Bathia-Davis inequality takes the form $2S_2\leq S_1(u-1)$ which is precisely the second order Fréchet inequality for $u>1$. If $\ell=u$, then both inequalities are equivalent to $S_2=S_1(S_1-1)$: the former because $0=\sigma^2=2S_2-S_1(S_1-1)$, the latter because 
$$2S_2=\Ea{M\atop 2}=u(u-1)=S_1(S_1-1)$$ 
as $p_u=1$ and $\ell=S_1=u$. Suppose now that $\ell>0$. In this case, the Bathia-Davis inequality takes the form $2S_2\leq S_1(\ell+u-1)-\ell u$. Now $S_1$ is positive because $\ell>0$. Hence we can divide both sides of the inequality by $S_1$ to obtain 
$$\frac{2S_2}{S_1}+1+\ell\left(\frac{u}{S_1}-1\right)\leq u$$ 
which implies the second order Fréchet inequality because $\ell\left(\frac{u}{S_1}-1\right)$ is positive.
\end{proof}

\section{The first two moments of $M_c(G)$}\label{sec:main}
Let $G$ be a graph with order $n$ and size $m$ and $\mathcal{V}_c$ be the set of all subsets of $c$ elements of $V(G)$. Consider the probability space $\Psi_c(G)$ whose sample space is $\mathcal{V}_c$ equipped with the uniform measure $p$. Hence $p(U)=c!/\kn{n}{c},\,\forall U\in \mathcal{V}_c$ and the $U$'s are equally likely. The random variable $M_c(G)$ is defined over this space as the size of the subgraph induced by a set $U\in\mathcal{V}_c$ drawn with probability $p(U)$.  Since $p$ is the uniform measure it is clear that $\pra{M_c(G)=k}=i_G(c,k)c!/\kn{n}{c}$ where $i_G(c,k)$ is the number of induced subgraphs of $G$ with order $c$ and size $k$. Notice that 
\begin{itemize}
	\item[--] for any graph $G$ and any $c\geq 2$, $i_G(c,0)$ is the number of independent sets of $c$-vertices;
	\item[--] for any non bipartite graph $G$ and any $c\geq 2$, $i_G(c,\frac{\kn{c}{2}}{2})$ is the number of cliques of $c$-vertices while if $G$ is bipartite and, for simplicity, $c$ is even, then $i_G(c,\frac{c^2}{4})$ is the number of balanced bicliques of $G$;
	\item[--] if $T$ is a tree, then  $i_T(c,c-1)$ is the number of subtrees of order $c$. 
\end{itemize}
The support $R_c(G)$ of $M_c(G)$ is contained in the interval $\{\ell(c),\ldots,u(c)\}$. As saw in Section \ref{sec:intro}, $\ell(c)\geq \max\{0,c-\alpha(G)\}$, $\alpha(G)$ being the independence number of $G$. A trivial upper bound on $u(c)$ is given by $\gamma(c)$ where $\gamma_c$ is either $\lfloor c/2\rfloor\lceil c/2 \rceil$, i.e., the size of the complete bipartite graph on $c$ vertices (if $G$ is bipartite), or $\gamma_c$ equals $\frac{c(c-1)}{2}$, i.e., the size of the complete graph on $c$ vertices. A more refined bound for the maximum of $M_c(G)$ comes from the Motzkin-Straus Theorem \cite{motzstrauss} which asserts that if $A$ is the adjacency matrix of $G$, then the maximum of the quadratic form $x'Ax/2$ over the standard simplex of $\R^n$ equals $\frac{1}{2}(1-\frac{1}{\omega})$, $\omega$ being the clique-number of $G$. Hence, if $z$ is the characteristic vector of a set of $c$ vertices of $V(G)$, then $c^{-1}z$ belongs to the standard simplex and we conclude that $u(c)\leq \frac{c^2}{2}\left(1-\frac{1}{\omega}\right)$. Therefore $u(c)$ is bounded from above by the smallest among $\gamma_c$ and the Motzkin-Straus bound. Notice that if the graph is bipartite, then the two bounds coincide. If $G$ is non bipartite, then $\gamma_c$ is smaller than the Motzkin-Straus bound if and only if $c\leq \omega-1$. Let us consider some examples.
\begin{example}\label{ex:1}\,\,
\begin{enumerate}[label={\rm(\alph*)}]
	\item\label{com:a} If $G\cong K_n$, then $R_c(K_n)=\{\gamma_c\}$ for all $c\in [n]$. Hence, $\pra{M_c(K_n)=\gamma_c}=1$; the random variable is therefore degenerate.
	\item\label{com:b} If $G\cong K_{1,n-1}$, then $R_c(K_{1,n-1})=\{0,c-1\}$ and $M_c(K_{1,n-1})$ attains 0 with probability $1-\frac{c}{n}$ while $M_c(K_{1,n-1})$ attains $c-1$ with probability $\frac{c}{n}$. Observe that $1-\frac{c}{n}={n-1\choose c}/{n\choose c}$ and $\frac{c}{n}={n-1\choose c-1}/{n\choose c}$. 
	\item\label{com:c} If $G\cong K_{d,d}$, $n=2d$ and, for simplicity, $c$ is even, then $R_c(K_{d,d})=\{0,c-1,2(c-2),\ldots,c^2/2\}$. In this case 
	$$\pra{M_c(K_{d,d})=st}=2\frac{{d \choose s}{d \choose t}}{{n \choose s+t}},$$
	namely, if $Z=\frac{c-\sqrt{c^2-4M_c(K_{d,d})}}{2}$, then the distribution of $Z$ conditioned on the event $(Z\leq c/2)$ is the same as the the distribution of $Z$ conditioned on the event $(Z\geq c/2)$  and both are hypergeometric with parameters $c$, $d$ and $2d$.  
	\item\label{com:d} $R_c(K_2^d)=\{0,1,\ldots,c\}$, where $c$ is an even positive number, $c\leq d$ and $K_2^d$ is the graph consisting of $d$ disjoint edges. For $k$ such that $0\leq k\leq d$ one has 
	$$\pra{M_c(K_2^d)=k}=\left(2d \atop c\right)^{-1}\left(d \atop k,c-2k\right)2^{c-2k}.$$
	To see this, observe that the probability on the left hand side can be interpreted as the probability that in a random coloring of the vertices by the colors red and blue, one sees  $k$ edges whose ends have color blue, $c-2k$ edges whose ends have different colors, and $d-c+k$ edges whose ends have color red.   
	The fact that the formula above defines a probability distribution reduces to check that  
	$$\left(2d \atop c\right)=\sum_{k=0}^{c/2}\left(d \atop k,c-2k\right)2^{c-2k},$$
	an identity which follows by observing that 
	$$\left(2d \atop c\right)=[X^c](1+X)^{2m}=[X^c](1+2X+X^2)^{m}=\sum_{k=0}^{c/2}\left(d \atop k,c-2k\right)2^{c-2k},$$
	where in the last equality we used the multinomial Theorem and $[X^c]Q(X)$ denotes the coefficient of $X^c$ in the polynomial $Q(X)$. 
\end{enumerate}
\end{example}
The following theorem proves in a clearer and shorter way formulas for $S_1(c)$ and $S_2(c)$ given in \cite{our}.  
\begin{theorem}\label{thm:main}
Let $G$ be a graph with order $n$ and size $m$ and let $d(v)$ denote the degree of vertex $v\in V(G)$. For an integer  $c$ such that $2\leq c\leq n$, the first and second binomial moment of $M_c(G)$ are given by the formulas
\begin{equation}\label{eq:bin_mom_1}
	S_1=\frac{\kn{c}{2}}{\kn{n}{2}}m
\end{equation}
and
\begin{equation}\label{eq:bin_mom_2}
	2S_2=\frac{\kn{c}{3}}{\kn{n}{3}}\left\{\frac{n-c}{n-3}\sum_{v\in V(G)}d(v)(d(v)-1)+\frac{c-3}{n-3}m(m-1)\right\}.
\end{equation}
Moreover, $S_2(c)$ can be computed in linear time. 
\end{theorem}
\begin{proof}
We use Lemma \ref{lem:frechet}. For $e\in E(G)$, $e=uv$, let $A_e$ be the event in $\Psi_c(G)$ which occurs whenever $u$ and $v$ belong to the randomly drawn subset $U\in \mathcal{V}_c$. In other words, $A_e$ is the event that edge $e$ is induced by $U$. Hence $M_c(G)$ is the number of events which occur among the $A_e$'s or, equivalently, $M_c(G)=\sum_{e\in E(G)}I(A_e)$, where $I(\cdot)$ is the indicator function. In view of Lemma \ref{lem:frechet}, to compute $S_1(c)$ and $S_2(c)$ we have just to compute $\pra{A_e}$ for all $e\in E(G)$ and $\pra{A_e\wedge A_f}$ for all pair of distinct edges $e$ and $f$ of $G$. The first task is easy to accomplish: the probability that $e$ is induced by $U$ is simply 
$${n-2 \choose c-2}/{n \choose c}=\frac{\kn{n}{c-2}}{(c-2)!}\frac{c!}{\kn{n}{c}}=\frac{\kn{c}{2}}{\kn{n}{2}},$$
namely the probability that $U$ contains two specified vertices. We obtain \eqref{eq:bin_mom_1} after summing over $E(G)$. Let us compute $\pra{A_e\wedge A_f}$ for two distinct edges $e$ and $f$ of $G$. We distinguish two cases: either $e$ and $f$ are adjacent, we write $e\sim f$ in this case, or they are not and we write $e\not \sim f$ accordingly. If $e\sim f$, then $e$ and $f$ span three vertices (we have tacitly assumed that the graph is simple). Therefore $A_e\wedge A_f$ occurs if and only if these three vertices belong to the randomly drawn subset $U$. Reasoning as before, the probability that $U$ contains three specified vertices is $\frac{\kn{c}{3}}{\kn{n}{3}}$. If $e\not\sim f$, then they span four vertices and, in this case, $\pra{A_e\wedge A_f}$ is $\frac{\kn{c}{4}}{\kn{n}{4}}$, namely the probability that $U$ contains four specified vertices.
It follows that
\[
\begin{split}
S_2(c)&=\sum_{e\sim f}\frac{\kn{c}{3}}{\kn{n}{3}}+\sum_{e\not\sim f}\frac{\kn{c}{4}}{\kn{n}{4}}\\
&=\frac{\kn{c}{3}}{\kn{n}{3}}\sum_{v\in V(G)}{d(v) \choose 2}+\frac{\kn{c}{4}}{\kn{n}{4}}\left[{m \choose 2}-{d(v) \choose 2}\right]
\end{split}
\]
and after some simple manipulations we arrive at \eqref{eq:bin_mom_2}. 
In view of \eqref{eq:bin_mom_1} and \eqref{eq:bin_mom_2}, the problem of computing the first two binomial moments of $M_c$ is  polynomial-time solvable as it amounts to computing the degrees of $G$ and this task can be trivially accomplished in $O(m)$ time. 
\end{proof}
Below, we list the first two binomial moments of $M_c(G)$ for the graphs in \ref{com:a} through \ref{com:d} of Example \ref{ex:1}. In the sequel, when $G$ is understood, we omit explicit reference to $G$. In particular, we write $M_c$, $R_c$, for $M_c(G)$ and $R_c(G)$, respectively.
\begin{example}[Example \ref{ex:1} cont'd]\label{ex:1cond} For ease of notation, we omit the dependence on $c$ when writing the binomial moments. Recall that for a positive integer $d$, $K_2^d$ denotes the isomorphism class of a graph of order $2d$ and size $d$ consisting of $d$ disjoint edges.
\begin{enumerate}[label={\rm(\alph*)}]
	\item\label{com:a1} $G\cong K_n$, $n\geq 1$, $0\leq c\leq n$. One has $$S_1=\frac{\kn{c}{2}}{2},\quad \text{and}\quad 2S_2=\kn{c}{3}+\frac{\kn{c}{4}}{4}.$$ 
	Note that $K_n$ has $\frac{\kn{n}{3}}{2}$ copies of $P_3$ and $\frac{\kn{n}{4}}{8}$ copies of $K_2^2$. Also notice that
	$$\frac{2S_2}{S_1}+1=\frac{\kn{c}{2}}{2},$$
	is the maximum size of an induced subgraph of order $c$.
	\item\label{com:b1} $G\cong K_{1,n-1}$, $n\geq 2$, $2\leq c\leq n$. One has $$S_1=\frac{\kn{c}{2}}{n},\quad \text{and}\quad 2S_2=\frac{\kn{c}{3}}{n}.$$ 
	Note that $K_{1,n-1}$ has no copy of $K^2_2$. Also notice that
	\begin{equation}\label{eq:star_ratio}
		\frac{2S_2}{S_1}+1=c-1
	\end{equation}
	is the maximum size of an induced subgraph of order $c$.
	\item\label{com:c1} $G\cong K_{d,d}$, $n=2d$ and, for simplicity, $c$ is even and $2\leq c\leq 2d-1$. One has
	$$S_1=\frac{\kn{c}{2}}{2}\frac{d}{2d-1},\quad \text{and}\quad 2S_2=\frac{\kn{c}{3}}{4}\frac{d}{2d-1}\left\{2+(c-3)\frac{d-1}{2d-3}\right\}.$$ 
	Notice that 
	$$\frac{2S_2}{S_1}+1=\frac{\kn{c}{2}}{4}+\left(\frac{(c-2)(c-3)}{4(2d-3)}\right).$$
	Hence, 
	$$c^2/4-\left\lceil\frac{2S_2}{S_1}+1\right\rceil$$
	is $O(c)$ for all integer $c$ such that $2\leq c\leq 2d-1$ while, if $c$ is $\Theta(d)$, then this difference is $o(1)$. 
	\item\label{com:d1} $G\cong K_2^d$, $n=2d$ and, for simplicity, $c$ is even with $2\leq c\leq 2d-1$. One has
	$$S_1=\frac{\kn{c}{2}}{2(2d-1)},\quad \text{and}\quad 2S_2=\frac{\kn{c}{4}}{4(2d-1)(2d-3)}.$$ 
	Note that $G$ has of copy of $P_3$ and that 
	$$\frac{2S_2}{S_1}+1=\frac{c-2}{2}\left(\frac{c-3}{2d-3}\right)+1.$$
	Hence, the difference between the maximum size $\frac{c}{2}$ of an induced subgraph of order $c$ and the value $\lceil\frac{2S_2}{S_1}+1\rceil$ is $o(c)$ if and only if $c$ is $\Theta(d)$.
	\end{enumerate}
\end{example} 
Notice, that even in these striking special cases, the computation of the first two binomial moments might be non trivial. Consider, for instance, the case of complete bipartite graphs: a direct computation of $2S_2$ (but already of $S_1$) as $2\Ea{M_c(G) \atop 2}$ asks for computing the sum 
$$\sum_{s\leq c/2}\frac{{d \choose s}{d \choose t}}{{n \choose c}}\left\{[s(c-s)]^2-s(c-s)\right\}.$$
Therefore, as a by-product, Theorem \ref{thm:main}, gives combinatorial identities for the first two moments of the random variables defined in \ref{com:c} and \ref{com:d} of Example \ref{ex:1}. Remark that the support of these random variables is not an interval of non negative integer numbers.
\section{Applications}\label{sec:res}
In this section we prove the results given in Section \ref{sec:intro} and discuss further application of Theorem \ref{thm:main}. Let us start with Theorem  \ref{thm:bh}.
\subsection{Proof of Theorem \ref{thm:bh}}
To prove the theorem we need the following folklore fact which we prove in a slightly less standard way (essentially avoiding the use of Cauchy-Schwarz or Jensen's inequality).
\begin{lemma}\label{lem:degree}	Let $G$ be a graph of order $n$ and size $m$. The degrees $d(v)$, $v\in V(G)$, of $G$ satisfy the following inequality 
	\begin{equation}\label{eq:std_degrees}
		\sum_{v\in V(G)}d(v)(d(v)-1)\geq \frac{4m^2}{n}-2m.
	\end{equation} 
Equality is attained in \eqref{eq:std_degrees} by the degrees of regular graphs.
\end{lemma}

\begin{proof}
	To prove \eqref{eq:std_degrees}, we observe that $\sum_{v\in V(G)}d(v)(d(v)-1)$ is $2na$ where $a$ is the second binomial moment of the degree distribution of $G$, viewed as a discrete random variable. The first binomial moment $b$ of this variable is simply the average degree $\frac{2m}{n}$ of $G$. Since the variance of any random variable is non negative (because it is the expectation of a non negative random variable) it follows that $2a-b(b-1)\geq 0$ which is precisely the inequality we wanted to prove. Equality in \eqref{eq:std_degrees} is attained precisely when the variance of the degree distribution is zero. Hence it is attained when all the degrees are equal.
\end{proof}	
\begin{theoremproof}{\ref{thm:bh}}
Let $\mathcal{G}$ be the set of graphs that qualify for the thesis. We show that the graphs in $\mathcal{G}$ are regular. Denote by $s$ the common size of the induced subgraphs of order $c$. Let $U$ be any subset of $V(G)$ of $c+1$ vertices and $H$ be the graph induced by $U$. Since $H-v$ has size $s$ for all $v\in V(H)$, it follows that $H$ is regular. Therefore, any $v\in V(G)$ has the same number $\nu$ of neighbours in each subset $U'$ consisting of $c$ vertices of $V(G)$. Hence $d_G(v){n-1\choose c-1}=\nu{n\choose c}$ and we conclude that $d_G(v)$ is independent of $v$. Next, the hypothesis is equivalent to the fact that the variance $\sigma^2$ of $M_c(G)$ is zero for each $G\in \mathcal{G}$, i.e., $2S_2-S_1(S_1-1)=0$ on $\mathcal{G}$. By Lemma \ref{lem:degree} it holds that $\sum_{v\in V(G)}d_{G}(v)(d_{G}(v)-1)=4\frac{m^2}{n}-2m$, $\forall G\in \mathcal{G}$. Plugging the last expression in $S_2$ as defined by \eqref{eq:bin_mom_2} yields a degree-2 equation in $m$ with rational coefficients and, as such, it has at most two real roots. One of the roots is trivially $m=0$ which corresponds to co-complete graphs. The other root is obtained by solving the linear equation in $m$ given by $\frac{2S_2}{S_1}+1=S_1$ where the division by $S_1$ is entitled by $m\not=0$. Since $m=\kn{n}{2}/2$ satisfies this equation, we conclude that the other root corresponds to complete graphs. 
\end{theoremproof}

\subsection{Bounds for the size of the Densest and the Lightest induced Subgraph via Fréchet and Bathia-Davis inequalities}
Theorem \ref{thm:fre_1} gives a lower bound on the size of the so-called $c$-densest subgraph. The problem of approximating this size appears to be very difficult (see \cite{feige,khot}). The related problem of approximating the size $\ell(c)$ of the so-called $c$-lightest (or $c$-sparsest) subgraph, is at least as difficult as the previous one \cite{remi}. Having easily computable bounds for these problems is thus desirable. Let us start by proving Theorem \ref{thm:fre_1}. As said in Section \ref{sec:intro}, Theorem \ref{thm:fre_1} follows quite straightforwardly after the second order Fréchet inequality in the form
$$\frac{2S_2(c)}{S_1(c)}+1\leq u(c)$$
by Theorem \ref{thm:main}. It remains to show that the bound is actually achieved only by the complete graphs, the stars and the graphs consisting of an edge and an independent set of $n-2$ vertices. Since by \ref{com:a1} in Example \ref{ex:1cond}, the bound is attained by the complete graphs, we have only to show that stars and graphs that induce exactly one edge are the only other graphs attaining the bound. To this end, we need the following two extremality results.
\begin{theorem}[\cite{idue}]\label{thm:idue}
	Let $G$ be a graph of order $n$ and $c$ be an integer such that $4\leq c\leq n-4$. If $\#R_c=2$, then $G$ is either a star, the disjoint union of an edge ad $n-2$ vertices or the complement of such graphs.
\end{theorem}
For a graph $G$ of order $n$ and size $m$ denote $\sum_{v\in V(G)}d(v)(d(v)-1)$ by $D_2(G)$. Also, recall that a graph with no triangle is said to be a \emph{triangle-free} graph.
\begin{lemma}\label{lem:maxdegreetrian}
Among the graphs of order $n$, $n\geq 4$, and positive size $m$, those that satisfy
\begin{equation}\label{eq:spec_fre}
\frac{2S_2(3)}{S_1(3)}+1= u(3)
\end{equation}
are either the complete graphs or the complete bipartite graphs or the graphs consisting of $t$ disjoint edges, $1\leq t\leq \lfloor n/2 \rfloor$. 
\end{lemma}
\begin{proof}
Let $G$ be a graph of order $n$ and size $m$ which satisfies \eqref{eq:spec_fre}. For shortness, $u=u(3)$ and $\ell=\ell(3)$. The possible values for $u$ are $0,1,2,3$. Since $m\ge1$, $u\not=0$. By Theorem \ref{thm:listin}, for any $c$, the second order Fréchet inequality is attained as an equality by $M_c$ if and only if $\#R_c=1$ or $\#R_c=2$ and $\ell(c)=0$. If $\#R_3=1$, then $G$ is complete: since any three vertices induce a triangle, it follows that any two vertices are adjacent. On the other hand, if $G$ is complete, then $\#R_3=1$ and we are done. We can then assume that $G$ is not complete and that $\#R_c=2$. We claim that $u\not=3$ necessarily holds when $c=3$. To see this, observe that $\ell=0$ implies that $G$ has an independent set $I$ of at least three vertices while $u=3$ implies that $G$ has a triangle $T$. Now $\#I\cap T\leq 1$. If $n=4$, then this is impossible because $\#T\cup I\geq 5$. Hence $n\geq 5$ and there are vertices $u,\,v\in I\setminus T$. If the neighbourhood of one of these two vertices, $u$, say, misses some vertex $w$ of $T$, then $u$, $w$ and any of the other vertices of $T$ induce either one or two edges contradicting that $R_3=2$. Therefore, there is a vertex $w\in T$ which is adjacent to both $u$ and $v$. These three vertices induce two edges still contradicting that $\#R_3=2$. We conclude that $u$ is either 1 or 2 and that $G$ is triangle-free. Suppose first that $u=1$. Since $S_1$ is always positive because $m$ is such, \eqref{eq:spec_fre} implies that $S_2=0$. Since by definition $S_2$ is proportional to $D_2(G)$, we conclude that $d(v)\leq 1$. Hence $G$ consists of disjoint edges as stated. Suppose now that $u=2$. The graphs that satisfy $u=2$ must be bipartite. To see this, observe that if $G$ contained an odd cycle, then $G$ would contain an induced odd cycle $C$ of length at least five. But then we could find three vertices of $C$ that induce exactly one edge (for instance, a vertex $u$ and any two adjacent vertices of $C-u$) contradicting that $\#R_3=2$. 
Also, $G$ must be a complete bipartite graph. Indeed, let $A$ and $B$ be any two parts of $V(G)$ such that $A\cup B=V(G)$ and all the edges of $G$ have one end in $A$ and the other in $B$. Since $u=2$, there are three vertices $u$, $v$ and $w$ such that $uv,\,uw\in E(G)$. Without loss of generality, $u\in A$, $v,w\in B$. If $G$ were not complete, then there would exist a vertex $z\in A$ and a vertex $z'\in B$ with $z$ and $z'$ not adjacent. Hence the vertices $u,\,z'$ and $v$ would induce exactly one edge, contradicting that $\#R_3=2$. The proof is thus completed.
\end{proof}    

\begin{theoremproof}{\ref{thm:fre_1}}
By Theorem \ref{thm:listin}, the second order Fréchet inequality is attained as an equality by $M_c$ if and only if $\#R_c=1$ or $\#R_c=2$ and $\ell(c)=0$. Theorem \ref{thm:fre_1} now follows by Theorem \ref{thm:idue} and Lemma \ref{lem:maxdegreetrian} for every $c$ such that $3 \leq c\leq n-4$. It remains to prove the theorem for the other feasible values of $c$, namely, for $c\in\{n-3,n-2,n-1\}$.
We claim that if $G$ satisfies the second order Fréchet inequality as an equality for such values of $c$, then $G$ is either a star on $n-1$ leaves or $G$ has exactly one edge. Let us show first that the independence number $\alpha(G)$ of $G$ equals $n-1$. Let $I$ be an independent set of $\alpha(G)$ vertices. Suppose that $\alpha(G)\leq n-2$. Since, $\ell(c)=0$, it holds that $\alpha(G)\geq c\geq n-3$. Hence, there are two vertices $u,\,v\in V(G)\setminus I$. Now $\{u\}\cup I$ and $\{v\}\cup I$ both induce at least one edge. Moreover, $u$ and $v$ must be both adjacent to every vertex of $I$: if $u$, say, did not have this property, then we could find vertices $w',\,w''\in I$  such that $w'$ is adjacent to $u$ but $w''$ is not; hence we could find $J\subseteq I\setminus \{w',w''\}$ such that $\#J=c-2$ and $\{u,w'\}\cup J$ and $\{u,w''\}\cup J$ induce a different number of edges contradicting that $\#R_c\leq 2$. Hence $u$ and $v$ must be adjacent to all vertices in $I$. However, this is impossible because, for arbitrary subsets  $I'$ and $I''$ of $I$ such that $\#I'=c-1$, $\#I''=c-2$,  the sets $\{u\}\cup I'$ and $\{u,v\}\cup I''$ would induce subgraphs with different sizes, again contradicting that $\#R_c\leq 2$. This contradiction shows that $\alpha(G)\geq n-1$ and we conclude that $\alpha(G)=n-1$ because $G$ has at least one edge. Let $I$ be an independent set of $n-1$ vertices and let $u\not\in I$. Either $u$ is adjacent to exactly one vertex in $I$, and we are done in this case, or $u$ has more than one neighbor in $I$ and, for the same reasons given above, $u$ must be adjacent to every vertex in $I$ and we conclude that $G$ is a star on $n-1$ leaves.         
\end{theoremproof}
By exploiting the Bathia-Davis inequality, we can achieve more.
\begin{theorem}\label{thm:fre_2}
Let $G$ be a graph of order $n$ and size $m$. Assume that $G$ is neither complete nor co-complete and let $c$ be an integer number such that $4\leq c\leq n-4$. If $\ell_\star$ and $u^\star$ are integers such that  $0\leq \ell_\star\leq \ell(c)\leq u(c)\leq u^\star$, then $G$ has a subgraph of order $c$ whose size is at least 
$$\left\lceil\min\left\{\frac{2S_2(c)}{S_1(c)}+1,\,S_1(c)+\frac{\sigma^2(c)}{S_1(c)-\ell_\star}\right\}\right\rceil$$
and a subgraph of order $c$ whose size is at most
$$\left\lfloor S_1(c)-\frac{\sigma^2(c)}{u^\star-S_1(c)}\right\rfloor,$$
where $\sigma^2(c)=2S_2(c)-S_1(c)(S_1(c)-1)$ and $\ell(c)$ and $u(c)$ are the minimum and the maximum of $M_c$ in $R_c$. The lower and upper bounds are attained precisely by the graphs listed in Theorem \ref{thm:idue} when $\ell_\star=\ell(c)$ and $u^\star=u(c)$, respectively. We can always choose $\gamma_c$ for $u^\star$ and $c-\alpha^\star$ for $\ell_\star$ where $\alpha^\star$ is any upper bound on the independence number of $G$. Moreover, given $u^\star$ and $\ell_\star$, the bounds can be computed in time linear in the size of $G$.
\end{theorem}
\begin{proof}
The theorem follows straightforwardly by inverting the Bathia-Davis inequality and resorting to Theorem \ref{thm:listin} and to Theorem \ref{thm:idue} for the tightness. For an illustration, consider the upper bound on $\ell(c)$ when $G$ is a star on $n-1$ leaves and $c\leq n-1$. In this case $u(c)=c-1$. By \ref{com:b1} in Example \ref{ex:1cond}, one has $$S_1(c)=\frac{\kn{c}{2}}{n},\quad \text{and}\quad  2S_2(c)=\frac{\kn{c}{3}}{n}.$$
Hence 
$$u(c)-S_1(c)=(c-1)\left(\frac{n-c}{n}\right)$$ 
and 
$$\frac{\sigma^2(c)}{u(c)-S_1(c)}=\frac{\kn{c}{2}}{n}=S_1(c).$$
Therefore, in this case, for all feasible values of $c$, the bound correctly attains the value 0.
\end{proof}
It might be worth observing that these ideas lead to a slightly less standard presentation of Mantel's Theorem: \emph{A triangle-free graph has at most $\frac{n^2}{2}$ edges and the bound is attained by the complete bipartite graphs with color classes of the same cardinality.} 

\noindent Let $G$ be a triangle-free graph. Let, as above, $D_2(G)=\sum_vd(v)(d(v)-1)$ and, for shortness, write $D_2$ in place of $D_2(G)$. One has $u(3)\leq 2$ because $G$ is triangle-free. Now, $2S_2=\frac{6}{\kn{n}{3}}D_2$ while $S_1=\frac{6}{\kn{n}{2}}m$. Hence 
$$\frac{2S_2}{S_1}=\frac{D_2}{(n-2)m}.$$
Therefore
$$u(3)\leq 2\Longleftrightarrow\frac{2S_2}{S_1}\leq 1\Longleftrightarrow D_2\leq (n-2)m.$$
The rightmost side in the equivalence above is usually obtained by writing $D_2$ as $\sum_{uv\in E(G)}\{d(u)+d(v)-2\}$ and observing that in triangle-free graphs $d(u)+d(v)\leq n$ for any pair of distinct adjacent vertices. This is actually the only (questionable) simplification we can claim over the standard proof of Mantel's Theorem. The rest of the proof proceeds in the standard way: use Lemma \ref{lem:degree} to bound $D_2$ from below; this yields the inequality $4\frac{m}{n}-2\leq n-2$ which is precisely the thesis of the Theorem.
\subsection{Counting trivial subgraphs and hypo/hyper-dense induced subgraphs }
In this section we provide upper bounds on the number of trivial subgraphs. Recall that these subgraphs are either independent sets or cliques. Let us start by proving Theorem \ref{thm:count_stable}. Recall that for a graph $G$, $D_2(G)$ stands for $\sum_{v\i V(G)}d(v)(d(v)-1)$. In what follows, we see $D_2(\cdot)$ as a function on graphs of given order and accordingly we omit the reference to $G$.  
\mybreak
\begin{theoremproof}{\ref{thm:count_stable}} The number of independent sets of $c$ vertices equals ${n \choose c}\pra{M_c=0}$. By the Chung-Erd\H{o}s inequality, $\pra{M_c=0}=1-\frac{S^2_1}{2S_2+S_1}$ where we omit the reference to $c$ because it is understood. Let us bound $\frac{S^2_1}{2S_2+S_1}$ from below in the case of trees of order $n$. This bound gives an upper bound on $\pra{M_c=0}$. To this end, write this expression as 
	$$S_1\left(1+\frac{2S_2}{S_1}\right)^{-1},$$
where the division by $S_1$ is entitled by the fact that $S_1>0$ as $S_1=\frac{\kn{c}{2}}{\kn{n}{2}}m=\frac{\kn{c}{2}}{n}$. Observe now that any upper bound on $\frac{2S_2}{S_1}$ yields an upper bound on $\pra{M_c=0}$. For every tree of order $n$, the ratio 
$$\frac{2S_2}{S_1}=\frac{c-2}{n-2}\frac{n-c}{n-3}\left(\frac{D_2}{n-1}\right)+\frac{c-2}{n-2}\frac{c-3}{n-3}(n-2)$$
depends only on $D_2$. Among trees, $D_2/(n-1)$ attains the absolute maximum $\frac{(n-1)(n-2)}{n-1}=(n-2)$ on the stars with $n-1$ leaves. Therefore, every tree on $n$ vertices satisfies $\frac{2S_2}{S_1}\leq c-2$ and, consequently  
$$S_1\left(1+\frac{2S_2}{S_1}\right)^{-1}\geq \frac{c}{n}$$ 
is satisfied by every tree. Moreover, equality is attained if ad only if the tree is a star. We conclude that in any tree of order $n$ the number $i(c,0)$ of independent sets of $c$ vertices satisfies the relations 
 $$i(c,0)={n \choose c}\pra{M_c=0}\leq {n \choose c}\frac{c}{n}={n-1 \choose c-1}.$$
Since there are $n+1$ independent sets with less than two vertices (the empty set of vertices and the singletons) and no tree has an independent set with more that $n-1$ vertices, it follows that the total number of independent sets in a tree equals $1+n+\sum_{c=2}^{n-1}i(c,k)$ which is therefore bounded by $2^{n-1}+1=1+n+\sum_{c=2}^{n-1}{n-1 \choose c-1}$.   
\end{theoremproof}
Theorem \ref{thm:count_stable} relies on Chung-Erd\H{o}s inequality which is a tail inequality of the form $\pra{M_c\in T(1)}$. The same method can be used to bound from above the number of all induced subgraphs of size $\ell(c)$ by complementing the lower bound on the tail $\pra{M_c\in T(\ell(c)+1)}$ given by Petrov's inequality. The same idea can be applied to bound from above the number of induced subgraphs of size $u(c)$. In this case, $\pra{M_c=u(c)}=\pra{M_c\in T(u(c))}$ and we can use the second order factorial moment inequality to bound $i(c,u(c))$ from above. Let us designate a subgraph as \emph{trivial in $G$} if its size is either $\ell(c)$ or $u(c)$. Also, for a class of graph $\mathcal{G}$, let us say that $G$ contains a $\mathcal{G}$-copy if the isomorphism class of some induced subgraph of $G$ is in $\mathcal{G}$. If it happens that for all $c$ the isomorphism classes of the trivial subgraphs in $G$ belong to $\mathcal{G}$, then $i(c,u(c))$ is the number of $\mathcal{G}$-copies in $G$ and we can use tail inequalities to bound from above such number. When $\mathcal{G}$ is the class of complete graphs and $u(c)=\frac{\kn{c}{2}}{2}$, the trivial subgraphs are just the cliques of $G$. If $G$ is a tree $T$, $\mathcal{G}$ is the class of trees, and $u(c)=c-1$, then the trivial subgraphs are the subtrees of $T$. When $G$ is a bipartite graph, $\mathcal{G}=\{K_{\lfloor \frac{d}{2}\rfloor,\lceil \frac{d}{2}\rceil} \ |\ d\in \mathbb{N}\}$ and $u(c)=\lfloor \frac{c}{2}\rfloor \lceil \frac{c}{2}\rceil$, the trivial subgraphs are the balanced bicliques of $G$. By complementation, we obtain similar notions for the trivial subgraphs of size $\ell(c)$. Let us illustrate the method in the easiest case, namely the case of subtrees of a given tree.
\begin{theorem}\label{thm:subtrees}
Any tree of order $n$ has at most $2^{n-1}$ subtrees and the bound is attained precisely when $T$ is a star. 
\end{theorem}
\begin{proof}
The proof is actually almost identical to the previous one. As above, the absolute maximum of $D_2$  when $T$ ranges over the trees of order $n$ equals $(n-1)(n-2)$ and it is attained by stars. This fact, after simple manipulation in \eqref{eq:bin_mom_2}, implies that $S_2(c)\leq \frac{\kn{c}{3}}{n}$ equality being attained by stars and only by stars. The number $i(c,c-1)$ of subtrees of order $c$ of $T$ is therefore bounded from above by ${n \choose c}\frac{2S_2(c)}{(c-1)(c-2)}={n \choose c}\frac{c}{n}$ because of the second order factorial moment inequality (inequality \ref{com:v} in Theorem \ref{thm:listin}). The latter expression equals ${n-1 \choose c-1}$. The thesis now follows by observing that the number of all subtrees of $T$ is $n+\sum_{c\geq 2}i(c,c-1)$ where the term $n$ counts the subtrees on one vertex.  
\end{proof}
Upper bounds on the number of trivial subgraphs can be given in a more general form that can also be useful for applications in Network Science. Consider indeed the following common scenario. You have a large network $G$, observe a set of vertices and want to decide whether these vertices form a \emph{community}. Although this is a vague notion, one of the most largely accepted criteria for declaring a set $C$ of $c$ vertices to be a community is that these vertices induce a hyper-dense subgraph of $G$. Here hyper-dense means that the (edge)-density of the observed subgraph is larger than the density of $G$. Since the density of $G$ is $2m/\kn{n}{2}$ while the density of $C$ is $2m(C)/\kn{c}{2}$, where $m(C)$ is the number of edges induced by $C$, it follows that $C$ is hyper-dense if and only if $m(C)\geq S_1(c)$, namely, if the number of edges induced by $C$ is at least as large as the expected value of $M_c(G)$. However, knowing that $C$ is a hyper-dense set is not particularly informative without any clue about $R_c(G)$. To overcome this lack of significance of the criterion, the authors \cite{our} proposed a randomized test based on the tails of $M_c(G)$: declare $C$ significantly hyper-dense at the level $\alpha$, $\alpha\in (0,1)$ if $\pra{M_c(G)\geq m(c)}\leq \alpha$, namely, if the probability of observing subgraphs, denser than the observed one by random chance, is below some fixed small value of $\alpha$. Clearly even an approximate computation of this probability is a very difficult task. Deciding whether $0$ or $\kn{c}{2}/2$ are in the support of $M_c(G)$ amounts to solve the decision version of INDEPENDENT SET and CLIQUE, respectively (and these problems are not even in the APX class). However, tail inequalities can be used to bound $\alpha$ from above. This approach was taken in \cite{our} (and satisfactory pursued in \cite{ours1}) where the authors use the one-sided Chebyshev (also known as Chebyshev-Cantelli inequality) applied to the random variable $Z_c(G)=(M_c(G)-S_1(c))/\sigma^2(c)$ to provide the bound
$$\pra{M_c(G)\geq t}\leq \min\left\{1,\frac{1}{\xi^2(t)+1} \right\}$$
were $\xi^2(t)=(S_1(c)-t)^2/\sigma^2(c)$. Here we propose to use the second order factorial moment inequality and Petrov's inequality to improve the bound on the tail. Using the fact that hyper-dense subgraphs in $G$ are hypo-dense subgraphs in the complement of $G$ the following bound, which is valid for $t\geq S_1(c)$, is at least as good as the one given in \cite{our},    
$$\pra{M_c(G)\geq t}\leq \min\left\{1,\frac{S_2(c)}{\kn{t}{2}}, 1-\frac{\left(S_1(c)-t+1\right)^2}{2\overline{S}_2(c)+\overline{S}_1(c)},\frac{1}{\xi^2(t)+1} \right\}$$
were $\overline{S}_i(c)$, $i=1,2$ the $i$-th binomial moment of $M_c(\overline{G})$, $\overline{G}$ being the complement of $G$. For instance by choosing $t=\kn{c}{2}/2$, the bound above is an upper bound on the number of cliques with $c$-vertices. In general, the bound above yields upper bounds on the number of trivial subgraphs of $G$ by conveniently specializing the choice of $G$ and $t$.    

\section{Concluding Remarks}\label{sec:coclusion}
In this paper we promoted the study of the second order binomial moments structure of $M_c(G)$ as a versatile tool both in theoretical and practical problems. Fréchet's identities, which we used up to the second order in Lemma \ref{lem:frechet}, are essentially equivalent to the Inclusion/Exclusion principle (\cite{sachov}). Correspondingly, the second order moment inequalities that we used throughout the paper, can be considered as the probabilistic version of the so-called Bonferroni's inequalities \cite{bopre,hoppeseneta,sachov}. The use of these inequalities was pioneered in \cite{erdkap}, \cite{oneil}, and \cite{bender} for asymptotic enumeration of combinatorial objects. For our purposes, a Bonferroni-type inequality of order $r$ is an inequality of the form
$$L(S_1,S_2,\ldots S_r,t)\leq \pra{M\geq t}\leq U(S_1,S_2,\ldots S_r,t)$$
where $M$ is a non negative random variable, $L(\cdot)$ ad $U(\cdot)$ are (possibly non linear) real valued functions, $t$ is a real non negative number and $(S_i)_i\geq 1$ is the sequence of the binomial moments of $M$. In this paper, all the applications we presented, rest on second order Bonferroni's inequalities. We believe, especially for computing lower bounds, that the knowledge of the entire sequence of binomial moments would be of valuable interest (both for theoretical and practical applications). We plan to give such a sequence in a future paper.


\begin{thebibliography}{10}



\bibitem{hararyetal}	
J. Akiyama, G. Exoo, F.Harary, \emph{The graphs with all induced subgraphs isomorphic},
Bull. Malaysian Math. Soc. 2(2) (1979) 43--44.

\bibitem{alonrtal}
N. Alon, M. Krivelevich, B. Sudakov, \emph{Induced subgraphs of prescribed size}, J. Graph Theory, 43 (2003) 239--251.

\bibitem{itredibada}	
E. Angel, R. Campigotto, C. Laforest, \emph{A new lower bound on the independence number of graphs}
Discrete Appl. Math., 161 (6) (2013) 847--852.

\bibitem{our}
N. Apollonio, P.G. Franciosa, D. Santoni, \emph{A novel method for assessing and measuring homophily in networks through second-order statistics}, Sci Rep 12, 9757 (2022). 

\bibitem{ours1}
N. Apollonio, D. Blankenberg, F. Cumbo, P.G. Franciosa, D. Santoni, \emph{Evaluating homophily in networks via HONTO (HOmophily Network TOol): a case study of chromosomal interactions in human PPI networks}, Bioinformatics, 39 (1) (2023) btac763, 
https://doi.org/10.1093/bioinformatics/btac763

\bibitem{idue}
M. Axenovich, B. J\'{o}zsef, \emph{Graphs Having Small Number of Sizes on Induced k‐Subgraphs}, SIAM J. Discret. Math. 21(1) (2007) 264--272. 

\bibitem{bender}
E.A. Bender, \emph{Asymptotic methods in enumeration} SIAM Rev. 16(4) (1974) 485--515.

\bibitem{bada}
R. Bhatia, C. Davis, \emph{A better bound on the variance}, American Mathematical Monthly, (Mathematical Association of America) 107(4) (2000) 353–357.

\bibitem{bopre}
E. Boros, A. Pr\'{e}kopa, \emph{Closed form two-sided bounds for probabilities that exactly $r$ and at least $r$ out of $n$ events occur}, Math. Oper. Res., 14 (1989) 317--342

\bibitem{bosak}
J. Bos\'{a}k, \emph{Induced subgraphs with the same order and size}, Math. Slovaca 33 (1983) 105--119.

\bibitem{cfm} 
N. Calkin, A. Frieze, and B. McKay, \emph{On subgraph sizes in random graphs}, Combinatorics, Probability and Computing 1 (1992) 123–134.

\bibitem{chungerdos}
K.L. Chung, P. Erd\H{o}s, \emph{On the application of the Borel–Cantelli lemma}, Trans. Amer. Math. Soc., 72(1) (1952) 179-–186. 



\bibitem{erdkap}
P. Erd\H{o}s, I. Kaplansky, \emph{The asymptotic number of latin rectangles}, Amer. J. 
Math., 68 (1946) 230--236. 

\bibitem{feige}
U. Feige, G. Kortsarz, D. Peleg, \emph{The Dense k-Subgraph Problem}, Algorithmica, 29(3) (2001) 410--421.

\bibitem{hoppeseneta}
F.M. Hoppe, E. Seneta, \emph{Gumbels's Identity, Binomial Moments, and Bonferroni Sums}, Int. Stat. Rev. 80(2) (2012) 269--292. 

\bibitem{khot}
S. Khot, \emph{Ruling out PTAS for graph min-bisection, dense k-subgraph, and bipartite clique}, SIAM J. Comput., 36(4) (2006) 1025-–107.

\bibitem{kp}
G. Kortsarz, D. Peleg, \emph{On choosing a dense subgraph}, Proceedings of 1993 IEEE 34th Annual Foundations of Computer Science, Palo Alto, CA, USA, 1993, pp. 692-701, doi: 10.1109/SFCS.1993.366818.

\bibitem{linlin}
S.-B. Lin and C. Lin, \emph{Trees and forests with large and small independent indices}, Chin. J. Math., 23(3) (1995), 199-210.

\bibitem{motzstrauss}
T.S. Motzkin, E.G. Straus, \emph{Maxima for graphs and a new proof of a theorem of Tur\'{a}n} Canad. J. Math. 17 (1965) 533-–540. 

\bibitem{oneil}
P.E. O'Neil, \emph{Asymptotics and random matrices with row-sum and column sum-restrictions}, Bull. Amer. Math. Soc., 75 (1964) 1276--1282.

\bibitem{petrov_0}
V.V. Petrov, \emph{On lower bounds for tail probabilities}, J. Statist. Planning and Inference, 137(8) (2007) 2703-2705. 

\bibitem{petrov}
V.V. Petrov, \emph{A generalization of the Chung-Erd\H{o}s inequality for the probability of a union of events}, J. Math. Sci. 147 (2007) 6932-–6934.


\bibitem{prodtichy}
H. Prodinger and R.F. Tichy, \emph{Fibonacci numbers of graphs}, Fibonacci Quarterly, 20(1) (1982), 16-21.

\bibitem{sachov}
V.N. Sachov, \emph{Probabilistic Methods in Combinatorial Analysis}, Encyclopedia of Mathematics
and its Applications, vol. 56, Cambridge University Press, 1997.

\bibitem{siran}
J. \v{S}ir\'{a}\v{n}, \emph{On graphs containing many subgraphs with the same number of edges}, Math. Slovaca 30 (1980) 267--268. 

\bibitem{remi}
R. Watrigant, M. Bougeret, R. Giroudeau, \emph{Approximating the Sparsest k-Subgraph in Chordal Graphs}, Theory Comput. Syst. 58(1) (2016) 111--132. 
\end{thebibliography}
\end{document}